\documentclass{article}
\usepackage{bussproofs,amsmath,amssymb,latexsym}
\usepackage{color}
\usepackage{graphicx}

\EnableBpAbbreviations

\newcommand{\colour}{\color}
\newcommand{\ljm}{{\text{ML}}}
\newcommand{\ljmm}{{\text{ML}^+}}
\newcommand{\exq}{\emph{ex falso sequitur quodlibet}}
\newcommand{\ex}{\emph{ex falso}}
\newcommand{\tnd}{\emph{tertium non datur}}
\newcommand{\dne}{\emph{double negation elimination}}
\newcommand{\oo}{\rightarrow}

\newcommand{\imp}{\supset}

\newtheorem{thm}{Theorem}
\newtheorem{lemma}{Lemma}
\newtheorem{cor}{Corollary}

\newenvironment{proof}{{\textbf{Proof.}} }{$\hfill \Box$}

\title{Relations between ex falso, tertium non datur, and double negation elimination}

\author{Pedro Francisco Valencia Vizca\'ino}

\date{March 30, 2013} 

\begin{document}

\maketitle

\begin{abstract}
We show which implicational relations hold between the three principles  \exq,
 \tnd, and \dne, on the basis of minimal logic.
\end{abstract}

\section{Introduction}

By \exq, or simply \ex\ for short, we mean intuitively the idea that we can derive anything from a contradiction. In slightly more formal terms, we mean a principle that allows us to derive all formulas   of the form 
$$\neg A \imp (A\imp B)$$ where $A$ and $B$ are arbitrary formul\ae.

By \tnd\ we mean a principle that allows us to derive all formulas of the form  $$A\lor \neg A$$
 where $A$ is an arbitrary formula, and by \dne\ we mean a principle that allows us to derive all   formulas of the form $$\neg \neg A \imp A.$$

The formal system we feel represents minimal logic is the sequent calculus system for intuitionistic logic LJ (see for example \cite{Tak}) without the rule weakening:right. 
We will argue that this rule really represents the logical principle \exq, and apply a standard cut elimination procedure to this and a slightly modified version of this system to show what relations hold between the logical principles \tnd, \ex, and \dne.

\section{Minimal Logic}

Let $\ljm$ stand for the formal system obtained from the formal system for intuitionist logic LJ  
by removing the rule weakening:right. We are thinking in the context of Sequent Calculus.
\vspace{3mm}

\begin{lemma}
Over $\ljm$ it is equivalent to have $\oo \neg A \imp (A\imp B)$ as initial sequents for all formulas $A$ and $B$ and to have the rule weakening:right.
\end{lemma}
\begin{proof}
If we add weakening:right to $\ljm$ we can prove  ex-falso quodlibet in the following straightforward way:

\begin{prooftree}
\AXC{$A\oo A$}
\UIC{$\neg A, A\oo$}\RightLabel{w:r}
\UIC{$\neg A, A\oo B$}
\UIC{$A, \neg A\oo B$}
\UIC{$\neg A\oo A\imp B$}
\UIC{$\oo\neg A\imp(A\imp B)$}
\end{prooftree}

Perhaps a bit less obvious is the fact that with this form of ex-falso quodlibet $\ljm$ proves \emph{Weakening-right}:

{\colour{blue}
\begin{prooftree}
\AXC{$A\oo A$}
\AXC{$A\oo $}
\UIC{$\oo \neg A$}
\UIC{$A\oo\neg A$}
\BIC{$A\oo A\land \neg A$}
\AXC{$\oo \neg A\imp(A\imp B)$}
\AXC{$\neg A\oo\neg A$}
\AXC{$A\oo A$}
\AXC{$B\oo B$}
\BIC{$A\imp B, A\oo B$} 
\BIC{$\neg A\imp (A\imp B), \neg A, A\oo B$}\RightLabel{Cut}
\BIC{$\neg A, A\oo B$}
\UIC{$A\land \neg A, A\oo B$}
\UIC{$A, A\land \neg A\oo B$}
\UIC{$A\land \neg A, A\land \neg A\oo B$}
\UIC{$A\land \neg A\oo B$}\RightLabel{Cut}
\BIC{$A\oo B$}
\end{prooftree}
}

\end{proof}

\begin{lemma}
The following hold on the basis of minimal logic:
\begin{enumerate}
\item Double negation elimination implies tertium non datur and ex falso.
\item Tertium non datur + ex falso imply double negation elimination.
\end{enumerate}
\end{lemma}
\begin{proof}

\begin{enumerate}

\item
The following is a derivation in minimal logic which shows how to get $A\lor \neg A$ from $\neg \neg (A\lor \neg A)\imp (A\lor \neg A)$:

\includegraphics[ bb = 155 545 350 712]{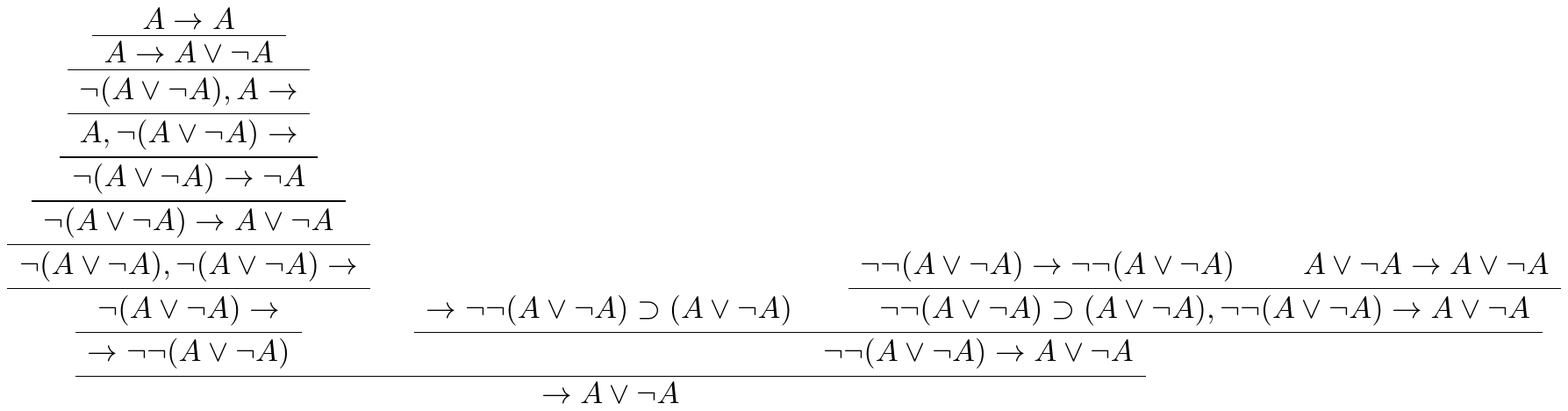}

And the next one shows that double negation elimination proves ex falso

\begin{prooftree}
\AXC{$\oo\neg\neg B\imp B$}
\AXC{$A\oo A$}
\UIC{$\neg A, A\oo$}   
\UIC{$\neg B, \neg A, A\oo$}
\UIC{$ \neg A, A\oo\neg\neg B$}
\AXC{$B\oo B$}
\BIC{$\neg\neg B\imp B,  \neg A, A\oo B$}
\BIC{$\neg A, A\oo B$}
\UIC{$A, \neg A\oo B$}
\UIC{$\neg A\oo A\imp B$}
\UIC{$\oo \neg A\imp(A\imp B)$}
\end{prooftree}

\item 
To see that with ex falso and tertium non datur we can obtain double negation elimination we use the following derivation:

\begin{prooftree}
\AXC{$\oo A\lor \neg A$}
\AXC{$A\oo A$}
\UIC{$\neg\neg A, A\oo A$}
\UIC{$A, \neg\neg A\oo A$}
\AXC{$\neg A\oo\neg A$}
\UIC{$\neg\neg A, \neg A\oo$}
\UIC{$\neg A,\neg\neg A\oo$}\RightLabel{{\colour{blue}{w:r}}} 
\UIC{$\neg A,\neg\neg A\oo A$}
\BIC{$A\lor\neg A, \neg\neg A\oo A$}
\BIC{$\neg\neg A\oo A$}
\end{prooftree}

\end{enumerate}

\end{proof}

\subsection{Minimal Logic and Tertium non datur}

The way we choose to include \tnd\ in $\ljm$ to get $\ljmm$ is by adding to $\ljm$ the following inference rule:

\begin{prooftree}
\AXC{$A, \Gamma\oo \Delta$}
\AXC{$\neg A, \Gamma\oo \Delta$}
\BIC{$\Gamma\oo\Delta$}
\end{prooftree}

It is easy to see that with this rule we can derive $A\lor \neg A$ for any formula $A$:

\begin{prooftree}
\AXC{$A\oo A$}
\UIC{$A\oo A\lor\neg A$}
\AXC{$\neg A\oo\neg A$}
\UIC{$\neg A\oo A\lor\neg A$}
\BIC{$\oo A\lor\neg A$}
\end{prooftree}

It is also easy to see that this rule is perfectly compatible with LK because $A\lor\neg A$ is always derivable in that system:

\begin{prooftree}
\AXC{$A\oo A$}
\UIC{$A\oo A\lor\neg A$}
\UIC{$\oo A\lor\neg A, \neg A$}
\UIC{$\oo A\lor\neg A, A\lor \neg A$}
\UIC{$\oo A\lor\neg A$}
\AXC{$A, \Gamma\oo \Delta$}
\AXC{$\neg A, \Gamma\oo \Delta$}
\BIC{$A\lor \neg A, \Gamma\oo \Delta$}
\BIC{$\Gamma\oo \Delta$}
\end{prooftree}

So we are somewhat justified in claiming this to be an addition of \tnd\ to $\ljm$ to get $\ljmm$.

\section{Cut Elimination}

We will do what Girard does in his book Proofs and Types \cite{Gir} to prove cut elimination for $\ljm$ and $\ljmm$:

\subsection{Reductions}

\begin{prooftree}
\AXC{$ A, \Gamma\oo$}
\UIC{$\Gamma\oo \neg A$}
\AXC{$\Pi\oo A$}
\UIC{$\neg A, \Pi\oo$}
\BIC{$\Gamma, \Pi\oo$}
\end{prooftree}

changes to

\begin{prooftree}
\AXC{$\Pi\oo A$}
\AXC{$ A, \Gamma\oo$}
\BIC{$\Gamma, \Pi\oo$}
\end{prooftree}

\begin{prooftree}
\AXC{$\Gamma\oo A$}
\AXC{$\Gamma\oo B$}
\BIC{$\Gamma\oo A \land B$}
\AXC{$B, \Pi\oo\Lambda$}
\UIC{$A\land B, \Pi \oo \Lambda$}
\BIC{$\Gamma, \Pi\oo \Lambda$}
\end{prooftree}

changes to 

\begin{prooftree}
\AXC{$\Gamma\oo B$}
\AXC{$B, \Pi\oo\Lambda$}
\BIC{$\Gamma, \Pi\oo \Lambda$}
\end{prooftree}

\begin{prooftree}
\AXC{$\Gamma\oo B$} 
\UIC{$\Gamma\oo A \lor B$}
\AXC{$A, \Pi\oo \Lambda$}
\AXC{$B, \Pi\oo \Lambda$}
\BIC{$A \lor B, \Pi \oo \Lambda$}
\BIC{$\Gamma, \Pi\oo \Lambda$}
\end{prooftree}

changes to

\begin{prooftree}
\AXC{$\Gamma\oo B$} 
\AXC{$B, \Pi\oo \Lambda$}
\BIC{$\Gamma, \Pi\oo \Lambda$}
\end{prooftree}

\begin{prooftree}
\AXC{$A, \Gamma\oo B$}
\UIC{$\Gamma\oo A\imp B$}
\AXC{$\Pi\oo A$}
\AXC{$B\oo \Lambda$}
\BIC{$A\imp B, \Pi\oo \Lambda $}
\BIC{$\Gamma, \Pi\oo \Lambda$}
\end{prooftree}

changes to

\begin{prooftree}
\AXC{$\Pi\oo A$}
\AXC{$A, \Gamma\oo B$}
\BIC{$\Pi, \Gamma\oo B$}
\AXC{$B\oo \Lambda$}
\BIC{$\Pi, \Gamma \oo \Lambda$}
\end{prooftree}

\begin{prooftree}
\AXC{$\Gamma\oo A(a)$}
\UIC{$\Gamma\oo \forall x A(x)$}
\AXC{$A(t), \Pi\oo \Lambda$}
\UIC{$\forall x A(x), \Pi\oo\Lambda$}
\BIC{$\Gamma, \Pi\oo \Lambda$}
\end{prooftree}

changes to

\begin{prooftree}
\AXC{$\Gamma\oo A(t)$}
\AXC{$A(t), \Pi\oo \Lambda$}
\BIC{$\Gamma, \Pi\oo \Lambda$}
\end{prooftree}

\begin{prooftree}
\AXC{$\Gamma\oo A(t)$}
\UIC{$\Gamma\oo \exists x A(x)$}
\AXC{$A(a), \Pi\oo \Lambda$}
\UIC{$\exists x A(x), \Pi \oo \Lambda$}
\BIC{$\Gamma, \Pi\oo \Lambda$}
\end{prooftree}

changes to

\begin{prooftree}
\AXC{$\Gamma\oo A(t)$}
\AXC{$A(t), \Pi\oo \Lambda$}
\BIC{$\Gamma, \Pi\oo \Lambda$}
\end{prooftree}

The degree $\partial (A)$ of a \emph{formula} is defined as:

\begin{itemize}
\item $\partial (A)  = 1$ for $A$ atomic
\item $\partial (A\land B) = \partial (A \lor B) = \partial (A \imp B) =\rm{max}(\partial (A), \partial (B)) + 1 $
\item $\partial (\neg A) = \partial (\forall x A(x)) = \partial (\exists x A(x)) = \partial (A) + 1$
\end{itemize}

The \emph{degree} of a \emph{cut rule} is defined to be the degree of the formula which it cuts.

The \emph{degree} $d(\pi)$ of a \emph{proof} is the least upper bound  (l.u.b.) of the set of degrees of its cut rules. So $d(\pi)=0$ iff $\pi$ is cut-free.

The height of a proof is 
\begin{itemize}
\item $h(\pi)=1$ if $\pi$ is an axiom
\item $h(\pi)={\rm{l.u.b.}}(h(\pi_1),h(\pi_2))+1$ if $\pi$ is a proof with a binary last inference rule whose premisses are proved by the subproofs $\pi_1$ and $\pi_2$.
\item $h(\pi)=h(\pi_0)+1$ if $\pi$ is a proof with a unary last inference rule (eg weakening left) and the subproof of the premiss of this inference rule is $\pi_0$.
\end{itemize}

{\textbf{Notation}} If $\Pi$ is a sequence of formulas then by $\Pi-C$ we mean $\Pi$ where an \emph{arbitrary} number of ocurrences of $C$ have been deleted.

 \begin{lemma}
 Let $C$ be a formula of degree $n$, and $\pi, \pi'$ be minimal logic proofs of $\Gamma\oo C$ and $\Pi\oo\Lambda$ of degrees $<n$. Then there is a minimal logic proof of degree $<n$ of 
 $\Gamma, \Pi-C\oo\Lambda$.
 \end{lemma}
 \begin{proof}
 By induction on $h(\pi)+h(\pi')$.
 
 If both $\pi$ and $\pi'$ are proofs of height one, ie, axioms we are in the following situation:
 
 \begin{prooftree}
 \AXC{$C\oo C$}
 \AXC{$D\oo D$}\noLine
 \BIC{}
 \end{prooftree}
 
 If $D\neq C$ then 
 
 \begin{prooftree}
 \AXC{$D\oo D$}\RightLabel{w:l}
 \UIC{$C, D\oo D$}
 \end{prooftree}
 
 gives what we wanted.
 
 If $D=C$ then either
 
 \begin{prooftree}
 \AXC{$C\oo C$}
 \UIC{$C, C\oo C$}
 \end{prooftree}
 
 or
 
 \begin{prooftree}
 \AXC{$C\oo C$}
 \end{prooftree}
 
 gives what we wanted.
 
 If $\pi$ is an axiom and $\pi'$ isn't then the situation is the following
 
 \begin{prooftree}
 \AXC{$C\oo C$}
 \AXC{$\pi'\upharpoonright$}
 \UIC{$\Pi\oo \Lambda$}\noLine
 \BIC{}
 \end{prooftree}
 
 and we want to conclude $C,\Pi-C\oo\Lambda$ which we can do by weakening from $\Pi\oo\Lambda$.
 
 If $\pi'$ is an axiom then the situation is
 
 \begin{prooftree}
 \AXC{$\pi\upharpoonright$}
 \UIC{$\Gamma\oo C$}
 \AXC{$D\oo D$}\noLine
 \BIC{}
 \end{prooftree}
 
 and again if $D\neq C$ then we get what we want from $D\oo D$ by weakening up to $\Gamma, D\oo D$.
 
 If $D=C$ then $\Gamma\oo C$ or $\Gamma, C\oo C$ give us what we want.
 
 If neither $\pi$ nor $\pi'$ are axioms then either both $\pi$ and $\pi'$ introduce $C$ in their last inference rule, or one of them doesn't. 
 
 If $\pi$ doesn't, and in particular this includes the rule
 
 \begin{prooftree}
 \AXC{$A, \Gamma\oo C$}
 \AXC{$\neg A, \Gamma \oo C $}
 \BIC{$\Gamma\oo C$}
 \end{prooftree}
 
 then, we use the induction hypothesis with the premiss(es) of the last inference rule of $\pi$ and $\Pi\oo\Lambda$, and then use the last inference rule of $\pi$. For example:
 
 \begin{prooftree}
 \AXC{$A, \Gamma\oo C$}
 \AXC{$\Pi\oo\Lambda$}\dashedLine
 \BIC{$A, \Gamma, \Pi-C\oo\Lambda$}
 \AXC{$\neg A, \Gamma\oo C$}
 \AXC{$\Pi\oo\Lambda$}\dashedLine
 \BIC{$\neg A, \Gamma, \Pi-C\oo \Lambda$}
 \BIC{$\Gamma,\Pi-C\oo\Lambda$}
 \end{prooftree}
 
 Now if both $\pi$ and $\pi'$ introduce $C$ in their last inferences then we should look at the reductions above. For example:
 
 If we had
 
 \begin{prooftree}
 \AXC{$\Gamma\oo A(t)$}
 \UIC{$\Gamma\oo \exists x A(x)$}
 \AXC{$A(a), \Pi\oo \Lambda$}
 \UIC{$\exists x A(x), \Pi \oo \Lambda$}
 \BIC{$\Gamma, \Pi\oo \Lambda$}
 \end{prooftree}

 then by induction hypothesis we would have a proof of degree $<n$ of $\Gamma, A(a), \Pi-C\oo \Lambda$. Then, by the substitution lemma, we would get (merely by substituting) $\Gamma, A(t), \Pi-C\oo\Lambda$. So things look like this:
 
 \begin{prooftree}
 \AXC{$\Gamma\oo A(t)$}
 \AXC{$\Gamma, A(t), \Pi-C\oo\Lambda$}\doubleLine
 \UIC{$A(t), \Gamma, \Pi-C\oo \Lambda$}
 \BIC{$\Gamma, \Gamma, \Pi-C\oo\Lambda$}
 \end{prooftree}

 \end{proof}

 \section{Main Results}
 
\begin{lemma}
Let $B$ be an atomic formula. If the sequent $\Gamma\oo B$ is provable in $\ljmm$ then there is a formula $F$ which has $B$ as a subformula not occurring in the scope of a negation such that $F\in \Gamma$.
\end{lemma}
\begin{proof}
By induction on the height of a cut-free proof of $\Gamma\oo B$.
If $\Gamma\oo B$ is an axiom then it is the axiom $B\oo B$. In this case $F=B$ works because $B$ is a subformula of $F$ and it does not appear in the scope of a negation ($F$ has no negation symbols at all).

Let $\pi$ be a proof of $\Gamma\oo B$:

If the last inference rule of $\pi$ is weakening left

\begin{prooftree}
\AXC{$\Gamma'\oo B$}
\UIC{$D, \Gamma'\oo B$}
\end{prooftree}

then by induction hypothesis we have a formula $F$ in $\Gamma'$ in which $B$ appears as a subformula and not within the scope of a negation. This same formula $F$ appears in the sequence $D, \Gamma'$.

If the last inference rule of $\pi$ was contraction left or exchange left we observe that it is very similar to the first case (weakening left).

If the last rule was our new rule

\begin{prooftree}
\AXC{$D, \Gamma\oo B$}
\AXC{$\neg D, \Gamma\oo B$}
\BIC{$\Gamma\oo B$}
\end{prooftree}

then we note that by induction hypothesis, in $\neg D, \Gamma \oo B$ we already have the necessary formula $F$; moreover, this formula cannot be $\neg D$ because since $B$ is atomic it cannot be equal to $\neg D$ and it also cannot appear as a subformula of $\neg D$ not inside the scope of a negation, because every proper subformula of $\neg D$ is inside the scope of the head of $\neg D$ which is a negation. This means that the formula $F$ is already in $\Gamma$.

The other cases seem easy to verify since we can't use any negation rules if we have the atomic formula $B$ on the right hand side of the sequent, and the other inference rules won't add a negation to any of the formulas in their premisses.

Therefore the statement follows.
\end{proof}

\begin{cor}
If $A$ and $B$ are atomic formulas then 
$$\ljmm\nvdash\neg A, A\oo B$$
\end{cor}

\begin{cor}
$\ljmm$ does not prove weakening:right.
\end{cor}

\begin{cor}
$\ljmm$ does not prove ex falso quodlibet.
\end{cor}

\begin{thm}
The following hold on the basis of minimal logic:
\begin{enumerate}
\item Double negation elimination implies Ex falso and Tertium non datur.
\item Ex falso + Tertium non datur imply Double negation elimination.
\item Ex falso does not imply Tertium non datur.
\item Tertium non datur does not impy Ex falso.
\item Tertium non datur does not imply Double negation elimination.
\end{enumerate}
\end{thm}

\end{document}